\documentclass[a4paper]{article}
\usepackage[utf8]{inputenc}

\usepackage{amsmath,amsthm,amssymb}
\usepackage{geometry}
\usepackage{color}
\usepackage{mathdots}
\usepackage{mathtools}

\usepackage{algorithm}
\usepackage{algpseudocode}

\usepackage{verbatim}
\usepackage{hyperref}
\usepackage{amsmath,amsthm,amssymb}
\usepackage{color}
\newtheorem{theorem}{Theorem}
\newtheorem{lemma}[theorem]{Lemma}
\newtheorem{remark}{Remark}
\newtheorem{counterex}[theorem]{Counterexample}
\newcommand{\R}{\mathbb R}
\newcommand{\C}{\mathbb C}
\newcommand{\ii}{\mathfrak i}
\DeclareMathOperator{\opvec}{vec}
\DeclareMathOperator{\re}{Re}
\DeclareMathOperator{\im}{Im}

\newcommand{\la}{\lambda}

\newcommand{\mS}{{\mathbb S}}
\newcommand{\cI}{{\cal I}}
\newcommand{\cQ}{{\cal Q}}

\title{Uniqueness of solution of systems of generalized Sylvester and conjugate-Sylvester equations}

\author{Fernando De Ter\'an\thanks{Universidad Carlos III de Madrid ({\tt fteran@math.uc3m.es})
}
\ and Bruno Iannazzo\thanks{Università degli Studi di Perugia ({\tt bruno.iannazzo@unipg.it})
}
}

\begin{document}
    \maketitle

\begin{abstract}
  We provide a characterization for a periodic system of generalized Sylvester and conjugate-Sylvester equations, with at most one generalized conjugate-Sylvester equation, to have a unique solution when all coefficient matrices are square and all unknown matrices of the system have the same size. We also present a procedure to reduce an arbitrary system of generalized Sylvester and conjugate-Sylvester equations to periodic systems having at most one generalized conjugate-Sylvester equation. Therefore, the obtained characterization for the uniqueness of solution of periodic systems provides a characterization for general systems of generalized Sylvester and conjugate-Sylvester equations.     

\bigskip

\end{abstract}

\section{Introduction}

We discuss the uniqueness of solutions of systems of {\em generalized Sylvester equations}
\begin{equation}\label{eq:1}
    AXB-CXD = E,
\end{equation}and {\em generalized conjugate-Sylvester equations}
\begin{equation}\label{eq:2}
    AXB-C\overline X D = E.
\end{equation}
The matrix $X\in\C^{m\times n}$ is the unknown, $A,C\in\C^{m\times m}$, $B,D\in\C^{n\times n}$, and $\overline M$ denotes the (entrywise) conjugate of the matrix $M$. More precisely, the systems we consider are of the form
\begin{equation}\label{generalsysteme}
         A_i X_{\alpha_i}^{s_i}B_i-C_i  X_{\beta_i}^{t_i}D_i=E_i,\quad i=1,\hdots,r,
\end{equation}
where $s_i,t_i\in\{1,C\}$, for all $i=1,\hdots,r$, with $X^C:=\overline X$, and $\alpha_i,\beta_i$ are positive integers. In other words, the unknowns $X_{\gamma_i}^{d_i}$ are either $X_{\gamma_i}$ or $\overline X_{\gamma_i}$, for $\gamma_i=\alpha_i,\beta_i$ and $d_i=s_i,t_i$, respectively. This notation follows the one in \cite{dipr19}.

Note that the system \eqref{generalsysteme} includes equations of the type $A\overline X B - C \overline X D = E$ and $A\overline X B-CXD=E$, which are generalized Sylvester and conjugate-Sylvester equations, respectively, just replacing the roles of $X$ and $\overline X$.

We assume that all unknown matrices have the same size, namely
$X_{\gamma_i}\in\C^{m\times n}$, for $\gamma_i=\alpha_i,\beta_i$ and for all $i=1,\hdots,r$, and that the coefficient matrices are square. Then, it must be
\begin{equation}\label{size}
    A_i,C_i\in\C^{m\times m}\quad\mbox{and}\quad B_i,D_i\in\C^{n\times n},\quad \mbox{for all $i=1,\hdots, r$}.
\end{equation}

The presence of conjugation in some of the unknowns can make the system \eqref{generalsysteme} non-linear.  Consider, for instance, the case $A=B=C=D=1$, and $E=0$ in \eqref{eq:2}, so that the equation reads $x-\overline x=0$. In this case $x=1$ is a solution, but $\ii x=\ii$ is not.

Equation \eqref{eq:2} is associated with the system of equations (see, for instance, \cite[\S 4.3]{hj-topics})
\begin{equation}\label{matrixsystem}
    (B^\top\otimes A - (D^\top \otimes C)P )\opvec (X) = \opvec(E),
\end{equation}
where $\opvec(M)$ denotes the {\em vectorizing} operator applied to the matrix $M$, which stacks in a column vector the columns of the matrix $M$ one on top of the other, from the first one to the last one (in this order), and $P:\C^{mn}\rightarrow\C^{mn}$ is the conjugation operator, namely such that $P\opvec (X) = \opvec(\overline X)$. The system \eqref{matrixsystem} is not $\C$-linear since $P$ is not $\C$-linear, but if we split the real and imaginary parts it can be written as an equivalent $\R$-linear system
\begin{equation}\label{rlinear}
\begin{bmatrix}
    \re(B^\top \otimes A - D^\top \otimes C) & \im(- B^\top \otimes A - D^\top \otimes C)\\
    \im(B^\top \otimes A - D^\top \otimes C) & \re(B^\top \otimes A + D^\top \otimes C)
\end{bmatrix}
\begin{bmatrix}
\opvec(\re(X))\\ \opvec(\im(X))
\end{bmatrix}
= 
\begin{bmatrix}
\opvec(\re(E))\\ \opvec(\im(E))
\end{bmatrix}.
\end{equation}

The equivalence with an $\R$-linear system makes the uniqueness of solution of \eqref{eq:2}
equivalent to the uniqueness of solution of the homogeneous equation
$A X B - C \overline X D = 0.$ 
A similar (but more involved) approach can be followed for systems like \eqref{generalsysteme}, so that the uniqueness of solution of this system is equivalent to the uniqueness of solution of
\begin{equation}\label{generalsystem}
         A_i X^{s_i}_{\alpha_i}B_i-C_i X^{t_i}_{\beta_i}D_i=0,\quad i=1,\hdots,r.
\end{equation}
Hence, we will focus on these last homogeneous systems.

We will outline a procedure to associate a given general system like \eqref{generalsystem} with equivalent irreducible systems (namely, such that they cannot be split in two systems with independent unknowns) which are periodic and with at most one generalized conjugate-Sylvester equation, in such a way that the original system has a unique solution if and only if the irreducible systems have a unique solution. Then, we will obtain a characterization for the uniqueness of solutions of the resulting systems, which are of the form
\begin{equation}\label{periodic-system}
\left\{\begin{array}{l}
A_i X_i B_i - C_i X_{i+1} D_i = 0,\qquad i=1,\ldots,r-1,\\
A_r X_r B_r - C_r X_1^s D_r = 0,
\end{array}\right.
\end{equation}
with $s\in\{1,C\}$. 

We first address the case where the system does not have any conjugate-Sylvester equation at all (in Section \ref{gensyst_sec}), and then the one in which it has exactly one conjugate-Sylvester equation (in Section \ref{conjugate_sec}). The main results of the paper are Theorem \ref{maingen_th} (for periodic systems of just generalized Sylvester equations) and Theorem \ref{mainconj_th} (for periodic systems with exactly one generalized conjugate-Sylvester equation). 

The characterizations for the uniqueness of solution provided in these two results are given in terms of the spectral information of some matrix pencils constructed from the coefficient matrices of the system, together with their conjugates (in the case of Theorem \ref{mainconj_th}). We want to emphasize that Theorem \ref{maingen_th} is an amended version of Theorem 5 in \cite{dipr19}, which presented an incorrect characterization for the uniqueness of solution of a periodic system of generalized Sylvester equations, following the developments in \cite{byers-rhee}.

We wish to point out a slight update with respect to the result in \cite{dipr19}: Theorems \ref{maingen_th} and \ref{mainconj_th} are valid when the unknown matrices are rectangular, namely $m\times n$, and the coefficient matrices satisfy \eqref{size}, unlike the result in \cite{dipr19}, which was only presented for both the unknowns and the coefficient matrices being all square and of the same size.

There is a large amount of literature dealing with (systems of) Sylvester-like equations over the complex field. The generalized Sylvester equation has been considered since, at least, the 1980's (see, for instance, \cite{epton}). The interest in the generalized conjugate-Sylvester equation seems to be more recent, but despite this, there are many references dealing with this equation, especially focused on numerical methods for approximating the solution, when it exists (see, for instance, \cite{wu-etal2010,wu-etal2011}). In \cite{dfks17} a characterization for the consistency of systems of generalized Sylvester-like equations (including generalized Sylvester and conjugate-Sylvester ones) was provided. Other references have addressed the solvability of conjugate-Sylvester equations like $X-C\overline XD=E$, $XB-C\overline X$, or $AX-C\overline XD=E$, and have provided expressions for the solution, using the so-called ``real representation" \cite{jw03,w13,wdy06,wwd09}. Some recent papers have addressed the numerical solution of systems of generalized Sylvester equations, using either direct methods \cite{dipr19} or iterative methods to approximate either the solution \cite{conjgrad} or the least squares solution \cite{conjgrad-periodic}. Iterative methods have also been considered on systems involving conjugation, where the existence of a unique solution is assumed in advance \cite{wu-etal-unique}. However, up to our knowledge, no explicit characterizations for the uniqueness of solution have been provided so far in the literature.

The paper is organized as follows. In Section \ref{basic_sec} we present the notation and basic definitions of the manuscript. Section \ref{reduction_sec} describes the procedure to associate a given general system \eqref{generalsystem} with irreducible periodic systems having at most one generalized conjugate-Sylvester equation. In Section \ref{gensyst_sec} we address the uniqueness of solution of systems of generalized Sylvester equations, whereas in Section \ref{conjugate_sec} we address the uniqueness of solution for such systems when there is exactly one generalized conjugate-Sylvester equation. In the final section \ref{conclusion_sec} we summarize the main contributions of the paper and outline a future line of research.

\subsection{Basic notation and definitions}\label{basic_sec}

We follow the nomenclature of \cite{dipr19}. In particular, we say that the system \eqref{generalsystem} is {\it nonsingular} if it has only the trivial solution, namely $X_{\alpha_i}=X_{\beta_i}=0$, for all $i=1,\hdots,r$. We also say that the system \eqref{generalsystem}, with $t$ different unknowns, is {\it reducible} if $k$ of these unknowns, with $0<k<t$, appear only in $h$ equations, with $0<h<r$, and the remaining $t-k$ unknowns appear only in the remaining $r-h$ equations. In other words, a reducible system can be partitioned into two systems with no common unknowns. A system is said to be {\it irreducible} if it is not reducible.

A {\em matrix pencil} $M+\la N$, where $M,N$ are matrices of the same size, is {\it regular} if $M$ and $N$ are square and the determinant of the pencil (which is a polynomial in $\la$) is not identically zero. When $M+\la N$ is regular, $\la_0\in\C$ is a {\em finite eigenvalue} of $M+\la N$ if $\det(M +\la_0 N)=0$, and $\infty$ is an eigenvalue of $M+\la N$ if $N$ is not invertible. We use the notation $\Lambda(M+\la N)$ to denote the {\em spectrum} (namely, the set of eigenvalues) of $M+\la N$.

By $I_n$ we denote the identity matrix of size $n$, where we omit the subindex if the size follows from the context, and $\ii$ denotes the imaginary unit (namely, $\ii^2=-1$).

\begin{remark}\label{reversal_remark}
    The {\em reversal} of the pencil $M+\la N$, denoted by ${\rm rev}(M+\la N)$, is the pencil $N+\la M$. Since $\det(N+\la M)=\la^n\det(M+\la^{-1}N)$ (where $M$ and $N$ are of size $n\times n$), then $M+\la N$ is regular if and only if $N+\la M$ is regular and, if $M+\la N$ is regular, then $\Lambda(N+\la M)=\Lambda(M+\la N)^{-1}:=\{\lambda^{-1}:\ \lambda\in\Lambda(M+\la N)\}$, understanding that $1/0=\infty$ and $1/\infty=0$.
\end{remark}

\section{Reduction to a periodic system with at most one generalized conjugate-Sylvester equation}\label{reduction_sec}

Following the developments in \cite[\S5]{dipr19}, we reduce the problem of the uniqueness of solution of \eqref{generalsystem} to the same problem for an irreducible system of the form \eqref{periodic-system}
having exactly $r$ unknowns, each one appearing exactly twice, and with at most one conjugate unknown. The procedure is outlined below.

Let $\mS$ be a system like \eqref{generalsystem}, and let $\{1,\hdots,r\}=\cI_1\cup\cdots\cup \cI_\ell$ be a partition of the set of indices. Then we denote by $\mS(\cI_j)$, for $j=1,\hdots,\ell$, the system of equations comprising the equations with indices in $\cI_j$. With this notation, the reduction to an irreducible system like \eqref{periodic-system} starting from the system \eqref{generalsystem} can be done based on the following facts (in the following, by ``unknown" we mean an unknown matrix $X_i$):
\begin{enumerate}
    \item[(F1)] There is a partition $\mathcal I_1\cup\cdots\cup\,\mathcal I_\ell$ of $\{1,\ldots,r\}$ such that, for each $j=1,\ldots,\ell$, the system $\mS(\cI_j)$ is irreducible \cite[Prop. 1]{dipr19}.

    The system $\mS$ is nonsingular if and only if the system $\mS(\cI_j)$ is nonsingular, for all $j=1,\hdots,\ell$ \cite[Prop. 2]{dipr19}.
    \item[(F2)] If \eqref{generalsystem} is irreducible and nonsingular, the number of unknowns is exactly $r$ \cite[Prop. 3]{dipr19}. 
    \item[(F3)] If \eqref{generalsystem} is irreducible and the unknown $X_j$ appears in only one equation of the system, say $A_d X^{s_j}_j B_d-C_d X_{k}^{s_k}D_d=0$, then \eqref{generalsystem} is nonsingular if and only if $A_d,B_d$ are invertible and the system obtained after removing this equation is nonsingular and irreducible (same if the unknown is $X_k$ with $C_d$ and $D_d$ instead) \cite[Th. 6]{dipr19}.
\end{enumerate}

Using the previous three facts, we can reduce the system \eqref{generalsystem}, that we denote by $\mS$, to a set of irreducible periodic systems following these steps:

\begin{enumerate}

\item[(S1)] Partition $\mS$ as indicated in (F1). By (F2), the number of equations and unknowns in each of the irreducible systems $\mS(\cI_j)$ coincide. Let us concentrate on one of these systems, and let us denote it, again, by $\mS$. 

\item[(S2)]  Remove the equations having an unknown that appears only once in $\mS$. We end up with a new irreducible system, $\widetilde\mS$, where the number of equations and unknowns is the same, and each unknown matrix appears exactly twice. By (F3), the system $\widetilde\mS$ is nonsingular if and only if $\mS$ is nonsingular.

For simplicity, let us denote by $r$, again, the number of equations and unknowns of the new system $\widetilde\mS$ (even though it does not necessarily have the same number of equations and unknowns as the original system $\mS$).
\item[(S3)] By re-labeling the unknowns if necessary, the system $\widetilde\mS$ can be written as
  \begin{equation}\label{periodicbis}
    \begin{array}{cccc}
        A_i X^{s_i}_iB_i-C_i X_{i+1}^{t_i}D_i&=&0,& i=1,\hdots,r-1,\\
        A_r X_r^{s_r}B_r-C_r X^{t_r}_1D_r&=&0,
        \end{array}
    \end{equation}
where $s_i,t_i\in\{1,C\}$ (see \cite[Sec. 5.3]{dipr19}).
\item[(S4)]\label{laststep} There is a system of the form
    \begin{equation}\label{periodicy}
    \begin{array}{cccc}
        \widetilde A_iY_i\widetilde B_i-\widetilde C_i Y_{i+1}\widetilde D_i&=&0,&i=1,\hdots,r-1,\\
        \widetilde A_rY_r\widetilde B_r-\widetilde C_r  Y_1^s\widetilde D_r&=&0,
        \end{array}
    \end{equation}
    with $s\in\{1,C\}$, such that \eqref{periodicy} is nonsingular if and only if \eqref{periodicbis} is nonsingular. Also, $Y_i$ is either $X_i$ or $\overline X_i$, for $i=1,\hdots,r$ (see \cite[Lemma 3]{dipr19}).
\end{enumerate}

The reason why the previous developments, which in \cite{dipr19} are suitable for a system of generalized Sylvester and $\star$-Sylvester equations, work for the system \eqref{generalsystem} as well is the following. In \cite{dipr19}, a system like \eqref{generalsystem} is also considered, but each unknown  $X_i^{s_i}$ or $X_i^{t_i}$ is either $X_i$ or $X_i^\star$, where $\star$ stands for either the transpose or the conjugate transpose. The arguments used in \cite{dipr19} to get a periodic irreducible system with at most one unknown of the form $X_i^\star$ (namely, like the system \eqref{periodic-system} with $X_1^\star$ in the last equation, instead of $\overline X_1$), which support the previous steps (S1)--(S4), are based on the following key property of the $\star$ operator: the unknown 
 $(X_i^{s_i})^\star$ is either $X_i^\star$ (if $X_i^{s_i}=X_i$) or $X_i$ (if $ X_i^{s_i}=X_i^\star$). This property is also satisfied by the conjugate operator, namely $\overline{X_i^{s_i}}$ is either $\overline X_i$ (when $X_i^{s_i}=X_i$) or $X_i$ (when $X_i^{s_i}=\overline X_i$).

Starting from a general system of the form \eqref{generalsystem}, after the steps described in (S1)--(S4), we end up with a set of periodic systems \eqref{periodicy} which are nonsingular if and only if the original system \eqref{generalsystem} is nonsingular. Therefore, from now on we will focus on a periodic system like \eqref{periodicy}. In Section \ref{gensyst_sec} we obtain a characterization for the uniqueness of solution of such a system when $Y_1^s=Y_1$, and in Section \ref{conjugate_sec} we do the same when $Y_1^s=\overline Y_1$.

 We want to emphasize that, in the case in which $m=n$, the procedure from \cite{dipr19} that we have just outlined to reduce the system \eqref{generalsysteme} can be used to reduce a system of the form
$A_iX_{\alpha_i}^{s_i}B_i-C_iX_{\beta_i}^{t_i}D_i=E_i$, $i=1,\ldots,r,$
where $s_i,t_i\in\{1,C,\top,*\}$, to a periodic system \eqref{periodic-system} in which $s=1,C,\top$, or $*$.
The four types of periodic systems can be dealt with Theorems \ref{maingen_th} and \ref{mainconj_th} (if $s=1,C$), and with Theorem 4 of \cite{dipr19} (if $s=\top,*$).

\subsection{On the sign of the second term}\label{sign_sec}

Let us consider the periodic system \eqref{periodic-system}. If we replace the minus signs in the second term of the equations by a plus sign, we get the system:
 \begin{equation}\label{newperiodic+}
    \begin{array}{cccc}
       A_iX_iB_i+ C_i X_{i+1} D_i&=&0,&i=1,\hdots,r-1,\\
         A_rX_r B_r+ C_r  X^s_1 D_r&=&0.
        \end{array}
    \end{equation}
It is not true, in general, that \eqref{periodic-system} is nonsingular if and only if \eqref{newperiodic+} is nonsingular. However, in the following Lemma \ref{plusminus_lem} we identify some situations in which this is true, depending on the parity of the number of equations, namely $r$. We prove, moreover, that, in these situations, the dimension of their solution spaces as a real vector space coincide.

\begin{lemma}\label{plusminus_lem}
 Let $A_i,C_i\in\C^{m\times m}$ and $B_i,D_i\in\C^{n\times n}$. If one of the following situations hold, then the systems \eqref{periodic-system} and \eqref{newperiodic+} have solution spaces with the same dimension:
 \begin{itemize}
     \item[\rm(i)] $r$ is even and $ X_1^s=X_1$ in the last equation.
     \item[\rm(ii)] $r$ is odd and $ X_1^s=\overline X_1$ in the last equation.
 \end{itemize}
\end{lemma}
\begin{proof}
    To prove (i) just note that, when $r$ is even and $ X_1^s=X_1$, then $(X_1,\hdots,X_r)$ is a solution of \eqref{newperiodic+} if and only if $(X_1,-X_2,\hdots,X_{r-1},-X_r)$ is a solution of \eqref{periodic-system}, and that $(X_1^{(1)},\hdots,X_r^{(1)}),\hdots,\allowbreak(X_1^{(\ell)},\hdots,X_r^{(\ell)})$ are linearly independent if and only if $(X_1^{(1)},-X_2^{(1)},\hdots,X_{r-1}^{(1)},-X_r^{(1)}),\hdots,$\break $(X_1^{(\ell)},-X_2^{(\ell)},\hdots ,X_{r-1}^{(\ell)},-X_r^{(\ell)})$ are linearly independent, for $\ell\geq1$.

    To prove (ii), note that, when $r$ is odd and $ X^s_1=\overline X_1$, then $(X_1,\hdots,X_r)$ is a solution of \eqref{newperiodic+} if and only if $\ii(X_1,-X_2,\hdots,-X_{r-1},X_r)$ is a solution of \eqref{periodic-system}, and that $(X_1^{(1)},\hdots,X_r^{(1)}),\hdots,\allowbreak(X_1^{(\ell)},\hdots,X_r^{(\ell)})$ are linearly independent if and only if $\ii(X_1^{(1)},-X_2^{(1)},\hdots,-X_{r-1}^{(1)},X_r^{(1)}),\hdots,$\break$\ii(X_1^{(\ell)},-X_2^{(\ell)},\hdots,-X_{r-1}^{(\ell)},X_r^{(\ell)})$ are linearly independent, for $\ell\geq1$.
\end{proof}

\section{Periodic systems of generalized Sylvester equations}\label{gensyst_sec}

The main result in this section is Theorem \ref{maingen_th}, where we provide a characterization for the uniqueness of solution of a periodic system \eqref{periodic-system} when $X_1^s=X_1$ in the last equation.

\begin{theorem}\label{maingen_th}
     Let $A_i,C_i\in\mathbb C^{m\times m}$ and $B_i,D_i\in\mathbb C^{n\times n}$, for $i=1,\hdots,r$, and $r>1$. The system
    \begin{equation}\label{system}
\left\{\begin{array}{cccc}
			A_i X_i  B_i -  C_i X_{i+1} D_i &= & 0, & i = 1, \ldots, r-1, \\
			 A_r X_r  B_r - C_r X_{1} D_r &=&  0
			\end{array}\right.
    \end{equation}
 has only the trivial solution if and only if the matrix pencils
\begin{equation}\label{pencils}
\left[\begin{array}{cccc}
\la C_{r}&A_r&&\\
&\la C_{r-1}&\ddots&\\
&&\ddots&A_{2}\\
A_1&&&\la C_1
\end{array}\right]\quad\mbox{and}\quad
\left[\begin{array}{cccc}
\la B_r&D_{r-1}\\
&\la B_{r-1}&\ddots&\\
&&\ddots&D_1\\
D_r&&&\la B_1
\end{array}\right]
\end{equation}
are regular and have disjoint spectra.  
\end{theorem}

In order to prove Theorem \ref{maingen_th} we will make use of Lemma \ref{tech_lem} below. It deals with {\it formal matrix products}, namely products of the form $\Pi:=M_rN_r^{-1}M_{r-1}N_{r-1}^{-1}\cdots M_1N_1^{-1}$, where $M_i,N_i$, for $1\leq i\leq r$, are complex square matrices of the same size, say $n\times n$. These products are defined even when some of the matrices $N_i$, for $1\leq i\leq r$, are not invertible. Actually, the only ingredients we need from a formal product are its {\it eigenvalues}. The eigenvalues of a formal matrix product can be obtained from the {\it periodic Schur decomposition} (introduced in \cite{bojanczyk1992periodic}, see \cite[Th. 1]{dipr19}). This decomposition guarantees that the formal product $\Pi$ has the same eigenvalues as the formal product of the form $R_rT_r^{-1}R_{r-1}T_{r-1}^{-1}\cdots R_1T_1^{-1}$, where the matrices $R_i$ and $T_i$, for $1\leq i\leq r$, are all upper triangular. Then we set:
\begin{equation}\label{evalsprod}
\lambda_i=\frac{(R_1)_{ii}(R_2)_{ii}\cdots(R_r)_{ii}}{(T_1)_{ii}(T_2)_{ii}\cdots(T_r)_{ii}},\qquad i=1,2,\hdots,n,
\end{equation}
with the convention that $\la_i=\infty$ if $(T_1)_{ii}(T_2)_{ii}\cdots(T_r)_{ii}=0\neq (R_1)_{ii}(R_2)_{ii}\cdots(R_r)_{ii}$. If there is some $1\leq i\leq r$ such that $(T_1)_{ii}(T_2)_{ii}\cdots(T_r)_{ii}=0= (R_1)_{ii}(R_2)_{ii}\cdots(R_r)_{ii}$ the formal product $\Pi$ is called {\it singular}, otherwise it is called {\it regular} (see \cite[p. 5--6]{dipr19} for the analogous notions on formal matrix products of the form $M_r^{-1}N_r\cdots M_1^{-1}N_1$). If $\Pi$ is regular, then $\la_1,\hdots,\la_n$ in \eqref{evalsprod} are the {\it eigenvalues} of $\Pi$.

Lemma \ref{tech_lem} is the counterpart of Lemma 1 in \cite{dipr19}, which deals with formal matrix products of the form $M_r^{-1}N_r\cdots M_1^{-1}N_1$ instead.

\begin{lemma}\label{tech_lem}
    Let $M_1,\hdots,M_r,N_1,\hdots,N_r\in\mathbb C^{n\times n}$ and let $\Pi=M_rN_r^{-1}\cdots M_1N_1^{-1}$ be a formal matrix product. Then, the matrix pencils
    $$
    \cQ_1(\la):=\begin{bmatrix}
        \la M_r&&&- N_1\\-N_r&\la M_{r-1}&&\\&\ddots&\ddots&&\\&&-N_2&\la M_1
    \end{bmatrix}\quad\mbox{and}\quad \cQ_2(\la)=\begin{bmatrix}
        \la N_r&-M_{r-1}\\&\la N_{r-1}&\ddots\\&&\ddots&-M_1\\-M_r&&&\la N_1
    \end{bmatrix}
    $$
    are regular if and only if $\Pi$ is regular. In this case, 
    \begin{itemize}
        \item[\rm(i)]  the eigenvalues of $\cQ_1(\la)$ are the inverses of the $r$th roots of the eigenvalues of $\Pi$, and
        \item[\rm(ii)] the eigenvalues of $\cQ_2(\la)$ are the $r$th roots of the eigenvalues of $\Pi$,
    \end{itemize} 
with the convention that $\sqrt[r]{\infty}=\infty.$
\end{lemma}
\begin{proof}
    Let us proceed as in the proof of Lemma 1 in \cite{dipr19} with $\cQ_1(\la)$.
    
    First, assume that all matrices $N_1,\hdots,N_r$ are invertible. In this case, the pencil $\cQ_1(\la)$ is regular and $0\not\in\Lambda(\cQ_1)$. Let $\la_0\neq\infty$ be an eigenvalue of $\cQ_1(\la)$. Then, there is a nonzero vector $u=\left[\begin{smallmatrix}
        u_1^\top&u_2^\top&\cdots&u_r^\top
    \end{smallmatrix}\right]^\top$  such that $\cQ_1(\la_0)u=0$, so 
    \begin{equation}\label{identities}
    \la_0M_ru_1=N_1u_r,\ \la_0M_{r-1}u_2=N_ru_1,\hdots,\ \la_0M_{2}u_{r-1}=N_{3}u_{r-2},\ \la_0M_1u_r=N_{2}u_{r-1}.
    \end{equation}
    If $u_r=0$, then the previous identities (starting from the rightmost one and moving leftwards), together with the fact that $N_1,\hdots,N_r$ are invertible, imply $u_{r-1}=\cdots=u_1=0$, which would imply $u=0$, contradicting the hypothesis, so it must be $u_r\neq0$. Now, by \eqref{identities}, $\la_0^rM_rN_r^{-1}\cdots N_{2}^{-1} M_1u_r=N_1u_r$, so $\la_0^rM_rN_r^{-1}\cdots N_{2}^{-1} M_1N_1^{-1}(N_1u_r)=N_1u_r$, namely, $N_1u_r$ (which is nonzero because $N_1$ is invertible) is an eigenvector of the formal product $\Pi$ associated with the eigenvalue $1/\la_0^r$.

    Conversely, let $\lambda_0$ be such that $\mu=1/\la_0^r$ is a finite eigenvalue of the formal product $\Pi$ and $u\neq0$ be such that $\Pi u=\mu u$. If we set
    $$
    u_r=N_1^{-1}u\ne 0,u_{j}=\la_0N_{r-j+1}^{-1}M_{r-j}u_{j+1},\ j=r-1,\hdots,1,
    $$    
    then $\la_0M_{r-j}u_{j+1}-N_{r-j+1}u_{j}=0$, for $j=1,\hdots,r-1$, and 
    $$
    \la_0M_ru_1-N_1u_r=\la_0^r M_rN_r^{-1}\cdots N_2^{-1}M_1N_1^{-1}u-N_1u_r=(1/\mu)\Pi u-u=0,
    $$
    which means that $\left[\begin{smallmatrix}
        u_1^\top&u_2^\top&\cdots&u_r^\top
    \end{smallmatrix}\right]^\top$ is an eigenvector of $\cQ_1(\la)$ associated with the eigenvalue $\la_0$.

    Note that $\infty\in\Lambda(\cQ_1)$ if and only if $0$ is an eigenvalue of $\Pi$, since in both cases this is equivalent to at least one of the matrices $M_1,\hdots,M_r$ being singular.

    We have proved, so far, that the eigenvalues of $\cQ_1(\la)$ coincide with the eigenvalues of the formal product $\Pi$ when all matrices $N_1,\hdots,N_r$ are invertible. To prove it when at least one of them is not invertible, as well as the claim that $\cQ_1(\la)$ is regular if and only if $\Pi$ is regular, we can proceed as in the last paragraph of the proof of Lemma 1 in \cite{dipr19}. More precisely, using the periodic Schur decomposition we can replace the matrices $M_1,\hdots,M_r$ and $N_1,\hdots,N_r$ by upper triangular matrices denoted, respectively, by $T_1,\hdots,T_r$ and $R_1,\hdots,R_r$. Then, the arguments in the last paragraph of the proof of Lemma 1 in \cite{dipr19}, together with the previous arguments, allow us to conclude that $\det\cQ_1(\la)=\pm\det(\la^r T_1\cdots T_r- R_1\cdots R_r)$, which immediately implies that $\cQ_1(\la)$ is regular if and only if $\Pi$ is regular, and that when they are both regular (including the case in which some of the matrices $N_1,\hdots,N_r$ are not invertible), the eigenvalues of $\cQ_1(\la)$ are the inverses of the $r$th roots of the eigenvalues of $\Pi$. This proves the first claim in the statement, together with part (i).

    For part (ii), note that $\cQ_2(\la)=-P\,{\rm rev}\cQ_1(\la)$, where $P$ is the block permutation matrix 
    $
    P=\left[\begin{smallmatrix} & 1\\ & & \ddots\\ & & & 1\\1\end{smallmatrix}\right]\otimes I_n,   
    $
   so the eigenvalues of $\cQ_2(\la)$ are the inverses of the eigenvalues of $\cQ_1(\la)$.
\end{proof}

\begin{proof} (of Theorem \ref{maingen_th}). The result is a consequence of Lemma \ref{tech_lem}, together with Lemma 1 and Theorem 2 in \cite{dipr19}. More precisely, by Theorem 2 in \cite{dipr19}, which is valid for $A_i,C_i\in\mathbb C^{m\times m}$ and $B_j,D_j\in\mathbb C^{n\times n}$ (the proof only requires the appropriate adjustments in the proof provided in \cite{dipr19}, namely replacing ``$i,j=1,\hdots,n$" by ``$i=1,\hdots,m$ and $j=1,\hdots,n$"), the system \eqref{system} has only the trivial solution if and only if the formal matrix products
$$
C_r^{-1}A_rC_{r-1}^{-1}A_{r-1}\cdots C_1^{-1}A_1\qquad\mbox{and}\qquad D_rB_r^{-1}D_{r-1}B_{r-1}^{-1}\cdots D_1B_1^{-1}
$$
are regular and have disjoint spectra. Now, by Lemma 1 in \cite{dipr19} and Lemma \ref{tech_lem} above, this is equivalent to the pencils
$$
\left[\begin{array}{cccc}
\la C_{r}&-A_r&&\\
&\la C_{r-1}&\ddots&\\
&&\ddots&-A_{2}\\
-A_1&&&\la C_1
\end{array}\right]\quad\mbox{and}\quad
\left[\begin{array}{cccc}
\la B_r&-D_{r-1}\\
&\la B_{r-1}&\ddots&\\
&&\ddots&-D_1\\
-D_r&&&\la B_1
\end{array}\right]
$$
being regular and having disjoint spectra. But this is in turn equivalent to the pencils in the statement being regular and having disjoint spectra, because, for $M,N$ being two matrices of the same size, the pencil $ M-\la N$ is regular if and only if $M+\la N$ is regular, and $\la_0$ is an eigenvalue of $M-\la N$ if and only if $-\la_0$ is an eigenvalue of $M+\la N$ (understanding that $-\infty=\infty$).
\end{proof}

The case $r=1$ is not considered in Theorem \ref{maingen_th}. In this case, the system reduces to
 $  AXB-CXD=0.$
A characterization for the uniqueness of solution of this equation was obtained in \cite{chu87} in terms of pencils associated with the coefficient matrices. The result in \cite{chu87} is only presented for real matrices, but the proof is still valid for complex matrices. We reproduce the result here for completeness and for the ease of comparison with Theorem \ref{maingen_th}.

\begin{theorem}{\rm \cite[Th. 1]{chu87}.}
    For $A,C\in\C^{m\times m}$ and $B,D\in\C^{n\times n}$, the equation $AXB-CXD=0$ has only the trivial solution if and only if the pencils
    $A-\la C$ and $D-\la B$ are regular and have disjoint spectra.
 \end{theorem}

\subsection{Other characterizations}\label{otherchar_sec}

By means of elementary block permutations, which preserve the regularity and the eigenvalues of matrix pencils, we can obtain other block-partitioned matrix pencils from the pencils in the statement of Theorem \ref{maingen_th}, which provide similar characterizations for the uniqueness of solution of \eqref{system}. More precisely, using
\begin{equation}\label{transform1}
\begin{bmatrix}
    &&&I\\&&I\\&\iddots\\
    I
\end{bmatrix}\begin{bmatrix}
    M_r&N_r\\&M_{r-1}&\ddots\\&&\ddots&N_2\\N_1&&&M_1
\end{bmatrix}\begin{bmatrix}
    &&&I\\&&I\\&\iddots\\
    I
\end{bmatrix}=\begin{bmatrix}
    M_1&&&N_1\\N_2&\ddots\\&\ddots&M_{r-1}\\&&N_r&M_r
\end{bmatrix},    
\end{equation}
together with
\begin{equation}
\begin{bmatrix}\label{transform2}
    &&I&\\&\iddots&\\I&&&\\&&&I
\end{bmatrix}\begin{bmatrix}
    M_r&N_{r-1}\\&\ddots&\ddots\\&&M_2&N_1\\N_r&&&M_1
\end{bmatrix}\begin{bmatrix}
    &&&I\\&&I\\&\iddots\\
    I
\end{bmatrix}=\begin{bmatrix}
    N_1&M_2\\&\ddots&\ddots\\&&N_{r-1}&M_r\\M_1&&&N_r
\end{bmatrix}    
\end{equation}
and
\begin{equation}\label{transform3}
\begin{bmatrix}
    &&&I\\&&I\\&\iddots\\
    I
\end{bmatrix} \begin{bmatrix}
    M_r&N_r\\&M_{r-1}&\ddots\\&&\ddots&N_2\\N_1&&&M_1
\end{bmatrix}\begin{bmatrix}
    I&&\\&&I\\&\iddots\\I
\end{bmatrix}=\begin{bmatrix}
    N_1&M_1\\&N_2&\ddots\\&&\ddots&M_{r-1}\\M_r&&&N_r
\end{bmatrix},
\end{equation}
we conclude, from Theorem \ref{maingen_th}, that \eqref{system} has a unique solution if and only if any of the following conditions hold:
\begin{itemize}
    \item[(i)] The pencils $\scriptsize\left[\begin{array}{@{\mskip3mu}c@{\mskip3mu}c@{\mskip3mu}c@{\mskip3mu}c@{\mskip3mu}}
        \la C_1&&&A_1\\A_2&\la C_2&\\&\ddots&\ddots\\&&A_r&\la C_r
    \end{array}\right]$ and $\scriptsize\left[\begin{array}{@{\mskip3mu}c@{\mskip3mu}c@{\mskip3mu}c@{\mskip3mu}c@{\mskip3mu}}
        \la B_1&&&D_r\\D_1&\la B_2&\\&\ddots&\ddots\\&&D_{r-1}&\la B_r
    \end{array}\right]$ are regular and have disjoint spectra.
    \item[(ii)] The pencils  $\scriptsize\left[\begin{array}{@{\mskip3mu}c@{\mskip3mu}c@{\mskip3mu}c@{\mskip3mu}c@{\mskip3mu}}
        A_1&\la C_1\\&A_2&\ddots\\&&\ddots&\la C_{r-1}\\\la C_r&&&A_r\end{array}\right]
    $
     and $\scriptsize
    \left[\begin{array}{@{\mskip3mu}c@{\mskip3mu}c@{\mskip3mu}c@{\mskip3mu}c@{\mskip3mu}c@{\mskip3mu}}
        D_1&\la B_2\\&D_2&\ddots\\&&\ddots&\la B_{r}\\\la B_1&&&D_r
    \end{array}\right]$ are regular and have disjoint spectra.
    \item[(iii)] The pencils $\scriptsize\left[\begin{array}{@{\mskip3mu}c@{\mskip3mu}c@{\mskip3mu}c@{\mskip3mu}c@{\mskip3mu}}
        \la A_1& C_1\\&\la A_2&\ddots\\&&\ddots& C_{r-1}\\C_r&&&\la A_r\end{array}\right]
    $
     and $\scriptsize
    \left[\begin{array}{@{\mskip3mu}c@{\mskip3mu}c@{\mskip3mu}c@{\mskip3mu}c@{\mskip3mu}}
       \la D_1&B_2\\&\la D_2& \ddots\\&&\ddots&B_{r}\\ B_1&&&\la D_r
    \end{array}\right]$ are regular and have disjoint spectra.
\end{itemize}
To get (i) we have applied the block row and block column permutations in \eqref{transform1} to both pencils in the statement of Theorem \ref{maingen_th}. To get (ii) we have applied the block row and block column permutations in \eqref{transform3} to the first pencil and the block row and block column permutations \eqref{transform2} to the second pencil in the statement of Theorem \ref{maingen_th}. Finally, to get (iii) we use Remark \ref{reversal_remark}.

\subsection{Theorem 5 in \cite{dipr19}}\label{br_sec}

In the prior work \cite[Th. 5]{dipr19} an incorrect characterization for the uniqueness of solution of \eqref{periodicsystem} is provided (see also \cite[Th. 3]{byers-rhee}). 

The claim there is that
if $A_i,B_i,C_i,D_i\in\mathbb C^{n\times n}$, for all $i=1,\hdots,r$, then the system
    \eqref{system} has only the trivial solution if and only if the matrix pencils
\begin{equation}\label{wrong-pencils}
\left[\begin{array}{cccc}
\la A_{1}&C_1&&\\
&\la A_{2}&\ddots&\\
&&\ddots&C_{r-1}\\
C_r&&&\la A_r
\end{array}\right]\quad\mbox{and}\quad
\left[\begin{array}{cccc}
\la D_{1}&B_1&&\\
&\la D_{2}&\ddots&\\
&&\ddots&B_{r-1}\\
B_r&&&\la D_r
\end{array}\right]
\end{equation}
are regular and have disjoint spectra. The following is a counterexample of this statement.

\begin{counterex}
Set $r=2$ and let $A_1=A_2=C_1=C_2=D_2=I_2$, $B_1=\left[\begin{smallmatrix}
        0&0\\1&0
    \end{smallmatrix}\right],B_2=\left[\begin{smallmatrix}
        0&1\\0&0
    \end{smallmatrix}\right]$, and $D_1=\left[\begin{smallmatrix}
        1&0\\0&0
    \end{smallmatrix}\right]$. Then the pencils in \eqref{wrong-pencils} read
    $$
    \begin{bmatrix}
        \la A_1&C_1\\C_2&\la A_2
    \end{bmatrix}=\begin{bmatrix}
        \la&0&1&0\\0&\la&0&1\\1&0&\la&0\\0&1&0&\la
    \end{bmatrix},\quad \mbox{and}\quad \begin{bmatrix}
        \la D_1&B_1\\B_2&\la D_2
    \end{bmatrix}=\begin{bmatrix}
        \la &0&0&0\\0&0&1&0\\0&1&\la&0\\0&0&0&\la
    \end{bmatrix}.
    $$
    Note that the pencils are both regular and have disjoint spectra, since $\det\left(\left[\begin{smallmatrix}
        \la A_1&C_1\\C_2&\la A_2
    \end{smallmatrix}\right]\right)=(\la^2-1)^2$ and $\det\left(\left[\begin{smallmatrix}
        \la D_1&B_1\\B_2&\la D_2
    \end{smallmatrix}\right]\right)=-\la^2$. However, the system \eqref{system} reads
    $$
    \begin{array}{l}
\left\{\begin{array}{ccc}
         \begin{bmatrix}
                x_1&x_2  \\
              x_3&x_4
         \end{bmatrix}\begin{bmatrix}
              0&0  \\
              1&0 
         \end{bmatrix}-\begin{bmatrix}
             y_1&y_2\\y_3&y_4
         \end{bmatrix}\begin{bmatrix}
             1&0\\0&0
         \end{bmatrix}& =&0, \\
         \ \begin{bmatrix}
                y_1&y_2  \\
              y_3&y_4
         \end{bmatrix}\begin{bmatrix}
             0&1\\0&0
         \end{bmatrix}-\begin{bmatrix}
             x_1&x_2\\x_3&x_4
         \end{bmatrix}&=&0, 
    \end{array}\right. \Longleftrightarrow \left\{\begin{array}{cc}
        x_1=x_3=0,  \\
        y_1=x_2,\ y_3=x_4, 
   \end{array}\right.\end{array}
    $$
    so the solution is not unique.

    If, instead, $B_1=\left[\begin{smallmatrix}
        0&1\\0&0
    \end{smallmatrix}\right]$ and $B_2=\left[\begin{smallmatrix}
        0&0\\1&0
    \end{smallmatrix}\right]$, with the rest of the coefficient matrices being as before, the pencils become:
     $$
    \begin{bmatrix}
        \la A_1&C_1\\C_2&\la A_2
    \end{bmatrix}=\begin{bmatrix}
        \la&0&1&0\\0&\la&0&1\\1&0&\la&0\\0&1&0&\la
    \end{bmatrix},\quad \mbox{and}\quad \begin{bmatrix}
        \la D_1&B_1\\B_2&\la D_2
    \end{bmatrix}=\begin{bmatrix}
        \la &0&0&1\\0&0&0&0\\0&0&\la&0\\1&0&0&\la
    \end{bmatrix},
    $$
    so that the second one is singular, and the associated system \eqref{system} reads
     $$
    \begin{array}{l}
\left\{\begin{array}{ccc}
         \begin{bmatrix}
                x_1&x_2  \\
              x_3&x_4
         \end{bmatrix}\begin{bmatrix}
              0&1  \\
              0&0 
         \end{bmatrix}-\begin{bmatrix}
             y_1&y_2\\y_3&y_4
         \end{bmatrix}\begin{bmatrix}
             1&0\\0&0
         \end{bmatrix}& =&0 \\
         \ \begin{bmatrix}
                x_1&x_2  \\
              x_3&x_4
         \end{bmatrix}\begin{bmatrix}
             0&0\\1&0
         \end{bmatrix}-\begin{bmatrix}
             y_1&y_2\\y_3&y_4
         \end{bmatrix}&=&0, 
    \end{array}\right.\Longleftrightarrow
   \left\{\begin{array}{ccc}
        \begin{bmatrix}
            -y_1&x_1\\-y_3&x_3
        \end{bmatrix}&=&0  \\
       \begin{bmatrix}
          x_2- y_1&-y_2\\x_4-y_3&-y_4
       \end{bmatrix} & =&0
   \end{array}\right.\\ \Longleftrightarrow \left\{\begin{array}{cc}
        x_1=x_2=x_3=x_4=y_1=y_2=y_3=y_4=0,
   \end{array}\right.\end{array}
    $$
    so it has a unique solution.
\end{counterex}

\section{Periodic systems with exactly one generalized conjugate-Sylvester equation}\label{conjugate_sec}

The main result in this section, which is the second main result in this work, is Theorem \ref{mainconj_th}, where we provide a characterization for the system \eqref{periodic-system} to have a unique solution when $X_1^s=\overline X_1$. This characterization is an immediate consequence of Theorem \ref{maingen_th} and the following Lemma. 

\begin{lemma}\label{doubled_lemma}
Let $A_i,C_i\in\C^{m\times m}$ and $B_i,D_i\in\C^{n\times n}$, for $i=1,\hdots,r$. The system
    \begin{equation}\label{periodicsystem}
    \left\{\begin{array}{cccc}
         A_iX_iB_i-C_iX_{i+1}D_i&=&0,&i=1,\hdots,r-1,  \\
         A_rX_rB_r-C_r\overline X_1D_r&=&0, 
    \end{array}\right.
    \end{equation}
    has only the trivial solution if and only if the system
    \begin{equation}\label{conjsystem}
         \left\{\begin{array}{cccl}
         A_iX_iB_i-C_iX_{i+1}D_i&=&0,&i=1,\hdots,r,  \\
         \overline A_iX_i\overline B_i-\overline C_i X_{i+1}\overline D_i&=&0,&i=r+1,\hdots 2r-1,\\
         \overline A_r X_{2r}\overline B_r-\overline C_r X_1\overline D_r&=&0,
    \end{array}\right.
    \end{equation}
    has only the trivial solution.
\end{lemma}
\begin{proof}
Note first that, if $(X_1,\hdots,X_r)$ is a nonzero solution of \eqref{periodicsystem}, then $(X_1,\hdots,X_r,\allowbreak\overline X_1,\hdots,\overline X_r)$ is a nonzero solution of \eqref{conjsystem}.

For the converse, if $X=(X_1,\hdots,X_r,X_{r+1},\hdots,X_{2r})$ is a nonzero solution of \eqref{conjsystem}, then $\widehat X:=(X_1+\overline  X_{r+1},\hdots,X_r+\overline X_{2r})$ is a solution of \eqref{periodicsystem}. If $\widehat X=0$, then $X_{r+1}=-\overline X_1$ and $\widetilde X:=(\ii X_1,\ldots, \ii X_r)$ is a solution of \eqref{periodicsystem}. Note that $\widetilde X$ is nonzero since, otherwise, $X_1=\cdots=X_r=0$ and $\widehat X=0$ implies that also $X_{r+1}=\cdots=X_{2r}=0$, contradicting that $X$ is nonzero.
\end{proof}

Just applying Theorem \ref{maingen_th} to the system of generalized Sylvester equations \eqref{conjsystem},
we obtain the following characterization for the uniqueness of solution of \eqref{periodicsystem}.

\begin{theorem}\label{mainconj_th}
Let $A_i,C_i\in\C^{m\times m}$ and $B_i,D_i\in\C^{n\times n}$, for $i=1,\hdots,r$. Then, the system
 \begin{equation*}
    \left\{\begin{array}{cccc}
         A_iX_iB_i-C_iX_{i+1}D_i&=&0,&i=1,\hdots,r-1,  \\
         A_rX_rB_r-C_r\overline X_1D_r&=&0, 
    \end{array}\right.
    \end{equation*}
    has only the trivial solution if and only if the pencils
    $$
    \footnotesize
    \left[\begin{array}{c@{\mskip3mu}c@{\mskip3mu}c@{\mskip3mu}c@{\mskip3mu}c@{\mskip3mu}c@{\mskip3mu}c}
        \la \overline C_r&\overline A_r\\&\ddots\ddots \\&&\la \overline C_1&\overline A_1\\&&&\la C_r& A_r\\&&&&\ddots\ddots\\&&&&&\la C_2&A_2\\A_1&&&&&&\la C_1
    \end{array}\right]\quad\mbox{\normalsize and}\quad
    \left[\begin{array}{c@{\mskip3mu}c@{\mskip3mu}c@{\mskip3mu}c@{\mskip3mu}c@{\mskip3mu}c@{\mskip3mu}c@{\mskip3mu}c}
        \la \overline B_r&\overline D_{r-1}\\&\ddots\ddots \\&&\la \overline B_2&\overline D_1\\&&&\la \overline B_1&D_r\\&&&&\la B_r& D_{r-1}\\&&&&&\ddots\ddots\\&&&&&&\la B_2&D_1\\\overline D_r&&&&&&&\la B_1
    \end{array}\right]
    $$
    are regular and have disjoint spectra.
\end{theorem}

Using the same approach as in Section \ref{otherchar_sec}, we can obtain similar characterizations for the uniqueness of solution of \eqref{periodicsystem}, namely, the system \eqref{periodicsystem} has only the trivial solution if and only if any of the following situations occur:
\begin{itemize}
    \item[(i)]  The pencils $$\footnotesize
    \left[\begin{array}{c@{\mskip3mu}c@{\mskip3mu}c@{\mskip3mu}c@{\mskip3mu}c@{\mskip3mu}c@{\mskip3mu}c}
        \la C_1&&&&&&A_1\\A_2&\la C_2 \\&\ddots&\ddots\\&&A_r&\la C_r\\&&& \overline A_1&\la\overline C_1\\&&&&\ddots\ddots\\&&&&& \overline A_{r}&\la \overline C_{r}
    \end{array}\right]\quad\mbox{\normalsize and}\quad
    \left[\begin{array}{c@{\mskip3mu}c@{\mskip3mu}c@{\mskip3mu}c@{\mskip3mu}c@{\mskip3mu}c@{\mskip3mu}c@{\mskip3mu}c}
        \la B_1&&&&&&&\overline D_r\\D_1&\la B_2 \\&\ddots&\ddots\\&&D_{r-1}&\la B_r\\&&& D_r&\la\overline B_1\\&&&&\overline D_1&\la\overline B_2\\&&&&&\ddots\ddots\\&&&&&& \overline D_{r-1}&\la \overline B_{r}
    \end{array}\right]
    $$
    are regular and have disjoint spectra. 
    \item[(ii)] The pencils $$\footnotesize
    \left[\begin{array}{c@{\mskip3mu}c@{\mskip3mu}c@{\mskip3mu}c@{\mskip3mu}c@{\mskip3mu}c@{\mskip3mu}c}
         A_1&\la C_1\\&\ddots\ddots \\&&A_r&\la C_r\\&&& \overline A_1&\la \overline C_1\\&&&&\ddots\ddots\\&&&&& \overline A_{r-1}&\la \overline C_{r-1}\\\la \overline C_r&&&&&& \overline A_r
    \end{array}\right]\quad\mbox{\normalsize and}\quad
    \left[\begin{array}{c@{\mskip3mu}c@{\mskip3mu}c@{\mskip3mu}c@{\mskip3mu}c@{\mskip3mu}c@{\mskip3mu}c@{\mskip3mu}c}
         D_1&\la B_2\\&\ddots\ddots \\&& D_{r-1}&\la B_r\\&&& D_r&\la \overline B_1\\&&&& \overline D_1&\la \overline B_2\\&&&&&\ddots\ddots\\&&&&&& \overline D_{r-1}&\la \overline B_{r}\\\la B_1&&&&&&& \overline D_r
    \end{array}\right]
    $$
    are regular and have disjoint spectra.
    \item[(iii)] The pencils $$\footnotesize
    \left[\begin{array}{c@{\mskip3mu}c@{\mskip3mu}c@{\mskip3mu}c@{\mskip3mu}c@{\mskip3mu}c@{\mskip3mu}c}
        \la A_1&C_1\\&\ddots\ddots \\&&\la A_r&C_r\\&&&\la \overline A_1&\overline C_1\\&&&&\ddots\ddots\\&&&&&\la \overline A_{r-1}&\overline C_{r-1}\\\overline C_r&&&&&&\la \overline A_r
    \end{array}\right]\quad\mbox{\normalsize and}\quad
    \left[\begin{array}{c@{\mskip3mu}c@{\mskip3mu}c@{\mskip3mu}c@{\mskip3mu}c@{\mskip3mu}c@{\mskip3mu}c@{\mskip3mu}c}
        \la D_1&B_2\\&\ddots\ddots \\&&\la D_{r-1}&B_r\\&&&\la D_r&\overline B_1\\&&&&\la \overline D_1&\overline B_2\\&&&&&\ddots\ddots\\&&&&&&\la \overline D_{r-1}&\overline B_{r}\\B_1&&&&&&&\la \overline D_r
    \end{array}\right]
    $$
    are regular and have disjoint spectra. 
\end{itemize}

\subsection{The case of a single conjugate-Sylvester equation}

It is interesting to consider independently a single generalized conjugate-Sylvester equation \eqref{eq:2}. 
Theorem \ref{mainconj_th} when $r=1$ is stated in the following result.

\begin{theorem}\label{singleconj_th}
Let $A,B\in\C^{m\times m}$, $C,D\in\C^{n\times n}$, and $E\in\C^{m\times n}$. Equation \eqref{eq:2} has a unique solution if and only if the pencils
$
\left[\begin{smallmatrix} 
\lambda \overline C & \overline A \\ A& \lambda C
\end{smallmatrix}\right]$ and $
\left[\begin{smallmatrix} 
\lambda \overline B &D \\ \overline D&\lambda B
\end{smallmatrix}\right]
$
are regular and have disjoint spectra.
\end{theorem}

Note that, since the coefficient matrix of the $\R$-linear system \eqref{rlinear} associated with Equation \eqref{eq:2} 
is square (of size $2mn\times 2mn$), Equation \eqref{eq:2} 
has a unique solution for some right-hand side $E$ if and only if it has a unique solution for every right-hand side $E$.

\begin{remark}\label{equivalence_rem}
The characterization for the uniqueness of solution of \eqref{eq:2} 
provided in Theorem {\rm\ref{singleconj_th}} is equivalent to the pencils
$
\left[\begin{smallmatrix} 
\lambda C & A \\ \overline A & \lambda \overline C
\end{smallmatrix}\right]$ and $
\left[\begin{smallmatrix}
\lambda \overline B &D \\\overline D &\lambda B
\end{smallmatrix}\right]
$
being regular and having disjoint spectra. This is so because 
$
\left[\begin{smallmatrix} 
\lambda C & A \\ \overline A & \lambda \overline C
\end{smallmatrix}\right]=\left[\begin{smallmatrix}
    0&I\\I&0
\end{smallmatrix}\right]\left[\begin{smallmatrix} 
\lambda \overline C & \overline A \\ A & \lambda C
\end{smallmatrix}\right]\left[\begin{smallmatrix}
    0&I\\I&0
\end{smallmatrix}\right],
$
and we use the fact that multiplying by constant invertible matrices preserves the regularity of matrix pencils, together with the eigenvalues.
\end{remark}

The uniqueness of solution of the standard conjugate-Sylvester equation 
$ AX-\overline XD=E$,
with $A\in\C^{m\times m}$ and $B\in\C^{n\times n}$, has been considered in \cite{bhh87,bhh88}. It is proved in \cite[Th. 3]{bhh87} that this equation has a unique solution if and only if the matrices $\overline AA$ and $\overline DD$ have disjoint spectra. This is a particular case of Theorem \ref{mainconj_th}. To see this, first note that the matrices $\overline AA$ and $\overline DD$ have disjoint spectra if and only if the formal matrix products $(-\overline A)I^{-1}(-A)I^{-1}$ and $(-\overline D)I^{-1}(-D)I^{-1}$ are regular and have disjoint spectra.  By Lemma \ref{tech_lem}, this is equivalent to the pencils 
$
\left[\begin{smallmatrix}\lambda I & A\\
\overline A& \lambda I\end{smallmatrix}\right]$ and $
\left[\begin{smallmatrix}\lambda I & D\\
\overline D & \lambda I\end{smallmatrix}\right]
$
being regular and having disjoint spectra, and, by Theorem \ref{singleconj_th} and Remark \ref{equivalence_rem}, this is in turn equivalent to the equation $AX-\overline XD=E$ having a unique solution.

\section{Conclusions and future work}\label{conclusion_sec}

We have obtained a characterization for a periodic system of generalized Sylvester and conjugate-Sylvester equations with at most one conjugate-Sylvester equation to have a unique solution when all unknown matrices have the same size and the coefficient matrices are square. We have also presented a procedure to reduce any system of generalized Sylvester and conjugate-Sylvester equations to  periodic systems with at most one conjugate-Sylvester equation, so that the provided characterization for the uniqueness of solution can be extended to such general systems.

A natural line of further research consists of looking for a characterization for the uniqueness of solution of systems of generalized Sylvester and conjugate-Sylvester equations when the unknown matrices have different sizes and the coefficient matrices are not necessarily square. The case of a single Sylvester equation (discussed in \cite{rectangular}) shows that the extension from square to rectangular coefficients is not immediate.

\bigskip

\noindent{\bf Acknowledgments}. We wish to thank Jie Meng, Federico Poloni and Leonardo Robol, for useful discussions on the topics of the paper. 

This research has been partially supported by grant PID2023-147366NB-I00 funded by MICIU/AEI/
10.13039/501100011033 and by FEDER/UE and RED2022-134176-T (Fernando De Ter\'an); by INdAM through a GNCS project, PRA 2022 WP 4.1 ``RATIONALISTS'' and ``AIDMIX'' funded by the University of Perugia, PRIN   2022ZK5ME7 CUP B53C24006410006 funded by Italian MUR and PNRR Vitality Spoke 9 (Bruno Iannazzo).

\end{document}